\documentclass[11pt]{amsart}

\usepackage{amsmath}
\usepackage{amssymb}
\usepackage{amscd}

\usepackage{enumitem}
\usepackage{mathtools}

\usepackage{tikz-cd}
\usepackage{bm}
\usepackage{ytableau}

\usepackage{biblatex}
\newcommand{\kk}{{\mathsf{k}}}

\newcommand{\PP}{{\mathbb{P}}}

\newcommand{\ZZ}{{\mathbb{Z}}}

\newcommand{\CA}{{\mathcal{A}}}
\newcommand{\CB}{{\mathcal{B}}}
\newcommand{\CC}{{\mathcal{C}}}

\newcommand{\CE}{{\mathcal{E}}}
\newcommand{\CF}{{\mathcal{F}}}
\newcommand{\CG}{{\mathcal{G}}}
\newcommand{\CH}{{\mathcal{H}}}

\newcommand{\CO}{{\mathcal{O}}}
\newcommand{\CP}{{\mathcal{P}}}

\newcommand{\CT}{{\mathcal{T}}}
\newcommand{\CU}{{\mathcal{U}}}
\newcommand{\CV}{{\mathcal{V}}}
\newcommand{\CW}{{\mathcal{W}}}

\newcommand{\RHom}{\mathop{\mathsf{RHom}}\nolimits}

\newcommand{\Hom}{\mathop{\mathsf{Hom}}\nolimits}
\newcommand{\Ext}{\mathop{\mathsf{Ext}}\nolimits}

\newcommand{\Gr}{{\mathsf{Gr}}}

\newcommand{\IGr}{{\mathsf{IGr}}}
\newcommand{\LGr}{{\mathsf{LGr}}}

\newcommand{\IFl}{{\mathsf{IFl}}}

\newcommand{\GL}{{\mathsf{GL}}}

\newcommand{\SP}{{\mathsf{Sp}}}

\newcommand{\lvee}[1]{\vphantom{#1}^\vee\!{#1}}
\newcommand{\YD}{\mathrm{Y}}
\newcommand{\bfG}{\mathbf{G}}
\newcommand{\sfP}{\mathsf{P}}
\newcommand{\CTE}{{\tilde{\mathcal{E}}}}

\DeclareMathOperator{\rk}{rk}

\theoremstyle{plain}

\newtheorem{theo}{Theorem}[]

\newtheorem{theorem}{Theorem}[section]

\newtheorem{lemma}[theorem]{Lemma}
\newtheorem{proposition}[theorem]{Proposition}
\newtheorem{corollary}[theorem]{Corollary}

\theoremstyle{definition}

\newtheorem{definition}[theorem]{Definition}

\theoremstyle{remark}

\newtheorem{remark}[theorem]{Remark}
\newtheorem{example}[theorem]{Example}

\title{Dual exceptional collections on Lagrangian Grassmannians}
\author{Anton Fonarev}
\address{\sloppy
\parbox{0.95\textwidth}{
Algebraic Geometry Section, Steklov Mathematical Institute of Russian Academy of Sciences,
8 Gubkin str., Moscow 119991 Russia
\hfill
}\bigskip}
\email{avfonarev@mi-ras.ru}
\date{}
\thanks{This work was supported by the Russian Science Foundation under grant no. 19-11-00164, https://rscf.ru/en/project/19-11-00164/}

\begin{document}

\begin{abstract}
    We construct graded left dual exceptional collections to the
    exceptional collections generating the blocks
    of Kuznetsov and Polishchuk
    on Lagrangian Grassmannians. As an application,
    we find explicit resolutions for some natural
    irreducible equivariant vector bundles.
\end{abstract}

\maketitle

\section{Introduction}
Derived categories of varieties are among the central objects
in modern algebraic geometry. A skeptical reader might complain
that the very notion of the derived category is too abstract,
and they might have a point. However, since the pioneering work
of Beilinson \cite{Beilinson1978}, derived categories have
become a great computational tool. Nowadays one would say that
Beilinson constructed a full exceptional collection in the bounded
derived categories of coherent sheaves on projective spaces.

The following
analogy is commonly used to explain the computational power of
exceptional collections. Given a finite dimensional real vector space $V$
with a positive definite symmetric bilinear form $\langle - ,- \rangle$,
one can find an orthonormal basis. That is, a basis consisting of vectors
$(v_1,v_2,\ldots,v_n)$ such that (i)~$\langle v_i, v_i\rangle = 1$ for all
$i=1,\ldots, n$ and (ii) $\langle v_i, v_j\rangle = 0$ whenever $i\neq j$.
The first condition says that the vectors are unit vectors, while
the second condition is the orthogonality condition.
Given such a basis, every vector $v\in V$ can be easily decomposed
with respect to our basis:
\begin{equation}\label{eq:orthonormal}
v = \langle v, v_1\rangle v_1 + \langle v, v_2\rangle v_2 + \cdots
+ \langle v, v_n\rangle v_n.
\end{equation}
Let us now relax some of the conditions;
for instance, symmetry. We still want a basis which consists of unit
vectors, but we have to modify the orthogonality condition (ii).
Let us say that a basis $(v_1,v_2,\ldots,v_n)$ is \emph{semiorthonormal} with respect
to a non-symmetric bilinear form $\langle - ,- \rangle$ if
(i)~$\langle v_i, v_i\rangle = 1$ for all $i=1,\ldots, n$ and
(ii) $\langle v_i, v_j\rangle = 0$ whenever $i > j$. One obviously needs
an adjustment to the formula~\eqref{eq:orthonormal}. Indeed,
in~\eqref{eq:orthonormal} we explicitly used the fact that under
the isomorphism $V^*\xrightarrow{\sim} V$ given by the form
$\langle - ,- \rangle$ the dual basis $(v^1,\ldots, v^n)$ maps
back to $(v_1,\ldots, v_n)$. Since we actually have two isomorphisms
this time, $V\to V^*$, $v\mapsto \langle v ,- \rangle$ and
$V\to V^*$, $v\mapsto \langle - ,v \rangle$, let us consider the second
one and denote by $u_i\in V$ the images of the dual basis its inverse.
That is, we have
\begin{equation*} % \label{eq:vect-dual}
    \langle v_i, u_j\rangle = \delta_{ij}\quad \text{for all } 1\leq i,j\leq n.
\end{equation*}
The vectors $u_1,\ldots, u_n$ obviously form a basis, and for any
$v \in V$ we have the desired formula
\begin{equation}\label{eq:vect-decomp}
    v = \langle v, u_1\rangle v_1 + \langle v, u_2\rangle v_2 + \cdots
+ \langle v, u_n\rangle v_n.
\end{equation}
What is less immediate, the vectors $(u_n, u_{n-1}, \ldots, u_1)$ (remark the reverse
order) form a semiorthonormal basis called the \emph{left dual} to
$(v_1,v_2,\ldots,v_n)$.

One rather indirect way to see that the latter holds is via \emph{mutations}.
Consider a semiorthonormal pair $(u, v)$: that is,
$\langle u, u \rangle=\langle v, v\rangle = 1$ and $\langle v, u \rangle=0$.
Let us define a new vector $\mathbb{L}_uv=v-\langle u, v\rangle u$, which
is called the \emph{left mutation} of $v$ through $u$. A simple calculation
shows that $(\mathbb{L}_uv, u)$ is a semiorthonormal pair. As the name suggests,
there is a sibling to the left mutation procedure: if one puts
$\mathbb{R}_vu = u-\langle u, v\rangle v$, which is called
the \emph{right mutation} of $u$ through $v$, then the pair
$(v, \mathbb{R}_vu)$ is semiorthonormal. It is a very nice exercise
in linear algebra to check that left and right mutations of adjacent
elements define an action of the braid group on $n$ strands on the set of all
semiorthonormal bases. Moreover, the left dual basis to a semiorthonormal
basis $(v_1,v_2,\ldots,v_n)$ is given by
\begin{equation}\label{eq:vect-mutation}
    \mathbb{L}_{v_1}\mathbb{L}_{v_2}\cdots\mathbb{L}_{v_{n-1}}v_n,\ 
    \mathbb{L}_{v_1}\mathbb{L}_{v_2}\cdots\mathbb{L}_{v_{n-2}}v_{n-1},\ 
    \ldots,\ 
    \mathbb{L}_{v_1}v_2,\ 
    v_1.
\end{equation}
We refer the reader to~\cite{Bondal2004} for further insights and
some interesting properties of this action.

The linear-algebraic picture translates to triangulated categories in the following
way. Instead of a~vector space we consider a $\kk$-linear triangulated category
$\mathcal{T}$ such that $\Ext^\bullet(E, F)$ is finite-dimensional
for all $E, F\in \mathcal{T}$ and treat $\Ext^\bullet_{\mathcal{T}}(-, -)$ as a kind of bilinear form.
Then one says that an object $E\in \mathcal{T}$ is \emph{exceptional}
if $\Hom(E, E)=\kk$ and $\Ext^i(E, E)=0$ for all $i\neq 0$.
A collection of objects $(E_1,E_2,\ldots, E_n)$ is called \emph{exceptional}
if every $E_i$ is an exceptional object and $\Ext^\bullet(E_i, E_j)=0$ for
all $i>j$. Finally, a~collection is called \emph{full} if it generates
the category in the sense that no proper strictly
full triangulated subcategory of $\mathcal{T}$ contains all $E_i$.
If $(E_1,E_2,\ldots, E_n)$ is a full exceptional collection in $\mathcal{T}$,
we will write $\mathcal{T}=\langle E_1,E_2,\ldots, E_n\rangle$.

The analogy with semiorthonormal bases should be clear by now.
Assume that $\mathcal{T}$ has a full exceptional collection
$(E_1,E_2,\ldots, E_n)$ (which is rarely the case).
What is most interesting is that not only every object in $\mathcal{T}$
can be obtained from the finite set of $E_i$'s by iteratively taking
shifts and cones, but this procedure can be made rather explicit.
Recall that the decomposition of every vector in terms of a given
semiorthonormal basis could be done with the help of a left
dual semiorthonormal basis by~\eqref{eq:vect-decomp}.
Let us mimic its definition in the categorical case.
\begin{definition}[\cite{Bondal1989a}]
    A collection of objects $(E_n^\vee,E_{n-1}^\vee,\ldots, E_1^\vee)$
    is \emph{left dual} to $(E_1,E_2,\ldots, E_n)$ if
    \begin{equation*}
        \Ext^\bullet(E_i, E_j^\vee)=0 \text{ for } i\neq j
        \quad\text{and}\quad
        \Ext^\bullet(E_i, E_i^\vee)=\kk \text{ for all } i=1,\ldots, n.
    \end{equation*}
\end{definition}
It should not surprise the reader at this point that
$(E_n^\vee,E_{n-1}^\vee,\ldots, E_1^\vee)$ is again
an exceptional collection, and it is full whenever the original one is.
As in the linear algebraic case, there are operations of left and right mutation
given for an exceptional pair $(E, F)$ by exact triangles
\begin{equation*}
    \mathbb{L}_EF\to \Ext^\bullet(E, F)\otimes E\xrightarrow{\mathrm{ev}} F\to \mathbb{L}_EF[1],
    \quad
    \mathbb{R}_FE[-1]\to E\xrightarrow{\mathrm{coev}} \Ext^\bullet(E, F)^*\otimes F\to \mathbb{R}_FE.
\end{equation*}
These define an action of the braid group on $n$ strands on the set
of all exceptional collections of length~$n$ in~$\mathcal{T}$,
and a formula similar to~\eqref{eq:vect-mutation} determines the left
dual collection:
\begin{equation*}
    E_i^\vee = \mathbb{L}_{E_1}\mathbb{L}_{E_2}\cdots\mathbb{L}_{E_{i-1}}E_i.
\end{equation*}
Finally, instead of the decomposition~\eqref{eq:vect-decomp}
one has a spectral sequence which computes cohomological functors applied
to objects in $\mathcal{T}$, which we will talk about
in Section~\ref{ssec:dual}.

The bridge between exceptional collections and semiorthonormal bases
is actually rather simple. Given $\mathcal{T}$ as above
with a full exceptional collection
$(E_1,E_2,\ldots, E_n)$, one checks that the classes of $E_i$
form a basis in the Grothendieck group $K_0(\mathcal{T})$. In particular,
the length of any full exceptional collection equals the rank of the latter,
which should a posteriori be a free finitely generated abelian group.
If one takes $\sum_i (-1)^i \dim_\kk \Ext^i(-, -)$ as the bilinear form,
exceptional collections become ``categorifications'' of the corresponding
semiorthonormal bases.

The first example of a full exceptional collection was given in~\cite{Beilinson1978},
in which Beilinson showed that
\begin{equation*}
    D^b(\PP^n)=\langle \mathcal{O}, \mathcal{O}(1),\ldots, \mathcal{O}(n)\rangle,
\end{equation*}
where $D^b(-)$ stands for the bounded derived category of coherent sheaves.
Interestingly enough, he also showed that the left dual (up to shifts)
is given by the collection
$\langle \Omega^n(n), \Omega^{n-1}(n-1),\ldots,\Omega^1(1),\mathcal{O} \rangle$.
A long-standing conjecture states that the bounded derived category of
coherent sheaves on a rational homogeneous variety admits a full exceptional
collection. Though a lot of work has been done over the years, the conjecture
has been
established in a very limited number of cases. Say, for classical groups
of Dynkin type $ABCD$ and homogeneous varieties of Picard rank $1$
the problem was fully resolved only for Grassmannians~\cite{Kapranov1984},
quadrics~\cite{Kapranov1988}, symplectic and orthogonal Grassmannians
of planes~\cite{Kuznetsov2008, Kuznetsov2021},
Lagrangian Grassmannians~\cite{Fonarev2022}, and in some sporadic cases.
We refer the reader to~\cite{Kuznetsov2016} for further details.

Since many explicit constructions realize varieties as subvarieties
in Grassmannians, exceptional collections on the latter
become an important computational
tool. One of our favorite recent examples can be found in~\cite{Lee2021},
where the authors use exceptional collections to study moduli of Ulrich
bundles. In order to use the tool's maximum power, it is important
to know the dual collection, and finding one is a task of its own.
In the present paper we find exceptional collections on Lagrangian Grassmannians
dual to those constructed in~\cite{Fonarev2022}. We further use them
to provide explicit resolutions of some very natural irreducible vector
bundles.

From now on we will be interested in the bounded derived category
of coherent sheaves on $\LGr(n, V)$, the Lagrangian Grassmannian of isotropic
subspaces of dimension $n$ in a fixed $2n$-dimensional vector space $V$
over a field $\kk$ of characteristic $0$ equipped with a non-degenerate
symplectic form. Exceptional collections of maximal length
(equal to the rank of the Grothendieck group) were constructed
on all symplectic and orthogonal Grassmannians by Kuznetsov and Polishchuk
in~\cite{Kuznetsov2016}. All these collections are conjecturally full;
however, the latter was checked only for Lagrangian Grassmannians in~\cite{Fonarev2022}.
The construction of Kuznetsov and Polishchuk is rather indirect: they start
with certain collections of irreducible vector bundles,
called \emph{blocks},
which naturally form an exceptional collection in the equivariant derived
category, then pass to dual collections within each block (again, in
the equivariant derived category). Each block must satisfy certain
homological conditions which guarantee that the dual collections become
exceptional in the non-equivariant derived category. Most of the hard work
in~\cite{Kuznetsov2016} is related to checking that certain collections
of irreducible equivariant vector bundles satisfy the block condition,
which is a rather difficult problem in representation theory.

While the block condition will be discussed in Section~\ref{ssec:fec-lgr},
let us explain why isotropic Grassmannians are much harder than the classical
ones. Full exceptional collections in the bounded derived categories
of Grassmannians were constructed by Kapranov~\cite{Kapranov1988},
who naturally extended Beilinson's method. Consider the Grassmannian
$\Gr(k, V)$ of $k$-dimensional subspaces in a fixed $N$-dimensional
vector space $V$ over a field $\kk$ of characteristic zero. Denote by
$\CU$ the tautological rank $k$ subbundle of the trivial bundle $V\otimes \CO$.
Kapranov showed that
\begin{equation}\label{eq:gr-fec}
    D^b(\Gr(k, V)) = \left\langle\Sigma^\lambda \CU^* \mid \lambda\in \YD_{k,N-k} \right\rangle,
\end{equation}
where $\lambda$ runs over the set of Young diagrams $\YD_{k, N-k}$
of height at most $k$ and width at most $N-k$,
$\Sigma^\lambda$ is the corresponding Schur functor, and the order
in this collection can be taken to be any linear order refining
the partial inclusion order $\subseteq$ on the diagrams. Moreover,
he simultaneously constructed the graded\footnote{Graded in the sense defined below.}
left dual to this collection. The latter
is given by
\begin{equation}\label{eq:gr-fec-dual}
    D^b(\Gr(k, V)) = \left\langle\Sigma^{\lambda^T} \CU^\perp \mid \lambda\in \YD_{k,N-k} \right\rangle,
\end{equation}
where $\CU^\perp = (V/\CU)^*$ and $\lambda^T$ denotes the transposed
diagram.
\begin{remark}
    The reader might wonder why we use this extra transposition
    and not index the collection by $\YD_{N-k, k}$, as Kapranov does.
    After all, $\{\lambda^T \mid \lambda \in \YD_{k, N-k}\} = \YD_{N-k, k}$.
    The reason for this choice becomes clear in Section~\ref{ssec:dual}, where we introduce
    our grading convention for dual exceptional collections.
\end{remark}

The duality relation might seem a little different from the one we described earlier:
\begin{equation}\label{eq:gr-dual}
    \Ext^\bullet(\Sigma^\lambda \CU^*, \Sigma^{\mu^T} \CU^\perp) =
    \begin{cases}
        \kk[-|\lambda|], & \text{if } \lambda=\mu, \\
        0, & \text{otherwise},
    \end{cases}
\end{equation}
but this grading difference does not change much since the two definitions
are equivalent up to shifts in the derived category. Actually, the grading
choice in~\eqref{eq:gr-dual} is favorable since the dual collection then
consists of vector bundles. Below, in Lemma~\ref{lm:gr-dual}, we give
a formal definition of a graded
dual exceptional collection, and this is an example of such a collection.

Return to the case when $V$ is symplectic of dimension $2n$.
Since $\LGr(n, V)$ is naturally embedded as a~closed subvariety in $\Gr(n, V)$,
and the tautological bundle on the latter restricts to the tautological bundle
on the former,
one can ask whether any of elements of Kapranov's collection restrict
to exceptional vector bundles on $\LGr(n, V)$. This is where surprising things
happen. It turns out that up to twist the only Young diagrams for which $\Sigma^\lambda \CU^*$
are exceptional on $\LGr(n, V)$ are $\lambda \in \YD_{n,1}$. The number of objects
in any full exceptional collection on $\LGr(n, V)$ equals
$\rk K_0(\LGr(n, V)) = 2^n$, and it is easy to see that one can not form
a sufficiently large exceptional collection using the latter bundles.
Remark also that $\CU^\perp$ restricts
to $\CU$ on $\LGr(n, V)$.

The construction of Kuznetsov and Polishchuk produces for any
$\lambda\in \YD_{h, n-h}$, where $0\leq h\leq n$, an~equivariant non-irreducible (in general)
vector bundle
$\CE^\lambda$ on $\LGr(n, V)$ with the following properties. First,
$\CE^\lambda$ belongs to the subcategory of the derived category $D^b(\LGr(n, V))$
generated by $\Sigma^\mu \CU^*$ for $\mu\subseteq \lambda$.
Second, the $\bfG=\SP_{2n}$-equivariant groups
$\Ext^\bullet_\bfG(\CE^\lambda, \Sigma^\mu \CU^*) = 0$ for all $\mu\subsetneq \lambda$,
while $\Ext^\bullet_\bfG(\CE^\lambda, \Sigma^\lambda \CU^*) = \kk$.
A given diagram may belong to various sets $\YD_{h, n-h}$, but the resulting
bundle does not depend on the choice of~$h$ since the previous conditions
fully characterize it and are independent of $h$. The first nontrivial example of such a bundle
is the universal extension
\begin{equation*}
    0\to \CO \to \CE^2 \to S^2\CU^* \to 0.
\end{equation*}

Kuznetsov and Polishchuk showed that for any $0\leq h \leq n$ the bundles
\begin{equation}\label{eq:lgr-block}
    \left\langle\CE^\lambda \mid \lambda\in \YD_{h,n-h} \right\rangle
\end{equation}
form an exceptional collection in $D^b(\LGr(n, V))$. As in the case of Grassmannians, one can
linearly order them in any way compatible with the partial inclusion
order on the corresponding Young diagrams. Let us denote by $\CF^\lambda$
the bundle dual (in the usual sense) to $\CE^\lambda$:
\begin{equation*}
    \CF^\lambda = (\CE^\lambda)^*.
\end{equation*}
Our first main result is the following theorem.

\begin{theo}[Theorem~\ref{thm:dual}]\label{thm:intro-dual}
    The bundles $\left\langle\CF^{\lambda^T} \mid \lambda\in \YD_{h,n-h} \right\rangle$
    form a graded left dual exceptional collection to~\eqref{eq:lgr-block}.
\end{theo}
We formulated Theorem~\ref{thm:intro-dual} in terms of transposed diagrams in order to show
the parallel between our Lagrangian situation and the case of classical
Grassmannians: compare the latter theorem with~\eqref{eq:gr-fec}
and~\eqref{eq:gr-fec-dual} keeping in mind that $\CU^\perp$ is isomorphic to
$\CU=(\CU^*)^*$ on $\LGr(n, V)$.

It is easy to see that the bundle $\Sigma^\lambda\CU^*$ belongs to
the subcategory~\eqref{eq:lgr-block} whenever $\lambda\in \YD_{h,n-h}$.
A nice geometric description (as a matter of fact, two of them)
was given for the bundles $\CF^\mu$ in~\cite{Fonarev2022}. Using this
description, one can compute the graded dual spectral sequences
and get the following result as an~application of Theorem~\ref{thm:intro-dual}
(see Section~\ref{ssec:res} for the details).

\begin{theo}[Theorem~\ref{thm:resol}]\label{thm:intro-resol}
    Let $\lambda\in \YD_{h,n-h}$ for some $0\leq h\leq n$. There
    is an exact $\SP(V)$-equivariant sequence of vector bundles on $\LGr(n, V)$ of the form
    \begin{equation*}
        0 \to \bigoplus_{\mu\in \mathrm{B}_{h(h+1)}}\CE^{\lambda/\mu}\to
        \cdots
        \to \bigoplus_{\mu\in \mathrm{B}_{2t}}\CE^{\lambda/\mu}
        \to \cdots\to
        \bigoplus_{\mu\in \mathrm{B}_4}\CE^{\lambda/\mu} \to
        \bigoplus_{\mu\in \mathrm{B}_2}\CE^{\lambda/\mu} \to
        \CE^\lambda\to \Sigma^\lambda\CU^*\to 0,
    \end{equation*}
    where $\mathrm{B}_{2t}$ denotes the set of balanced diagrams with $2t$ boxes,
    and $\CE^{\lambda/\mu}$ is a direct sum of certain bundles of the~form
    $\CE^{\nu}$, see Section~\ref{ssec:skew}.
\end{theo}

The paper is organized as follows. In Section~2 we collect all the preliminaries.
It contains no new material except, maybe, our convention for dual exceptional
collections for exceptional collections indexed by a graded partially ordered
set (see Lemma~\ref{lm:gr-dual}). Section~3 contains our main results.
That is, we prove Theorems~\ref{thm:intro-dual} and~\ref{thm:intro-resol},
which are Theorems~\ref{thm:dual} and~\ref{thm:resol}, respectively.

\section{Preliminaries}
Throughout the paper we work over a field $\kk$ of characteristic zero.

\subsection{Dual exceptional collections}\label{ssec:dual}
In the present section we collect some preliminaries related to (dual)
exceptional collections. The material is well known to specialists;
however, since we naturally work with exceptional collections indexed
by graded partially ordered sets, we introduce a certain convention
in the definitions.
This convention seems rather natural and useful, as we show in various examples.

\subsubsection{Partially ordered sets}
Recall that a \emph{partially ordered set (poset)} $\CP$ is a set equipped with
a binary relation $\preceq$, called a \emph{partial order,} satisfying
the following three properties:
\begin{description}
    \item[Reflexivity] $x\preceq x$ for all $x\in\CP$.
    \item[Antisymmetry] if $x\preceq y$ and $y\preceq x$, then $x=y$.
    \item[Transitivity] is $x\preceq y$ and $y\preceq z$, then $x\preceq z$ for all $x,y,z\in\CP$.
\end{description}
Elements $x$ and $y$ are called \emph{comparable} if either $x\preceq y$ or $y\preceq x$.
If every two elements in $\CP$ are comparable, one calls $\CP$
\emph{linearly ordered}. If $x\preceq y$ and $x\neq y$, one usually writes $x \prec y$.
If $\CP$ is partially ordered, its \emph{dual} $\CP^\circ$ is the set underlying
$\CP$ equipped with the converse relation: $x\preceq y$ in $\CP^\circ$ if and only
if $y\preceq x$ in $\CP$.

Let $x$ and $y$ be elements of a poset $\CP$. One says that $y$ \emph{covers} $x$,
written $x\lessdot y$, if $x\prec y$ and there is no element $z$ such that $x\prec z\prec y$.
A \emph{grading function} on $\CP$ is a map $\rho: \CP\to \ZZ$ with the following
properties:
\begin{itemize}
    \item if $x\prec y$ then $\rho(x) < \rho(y)$,
    \item if $x\lessdot y$ then $\rho(y)=\rho(x)+1$.
\end{itemize}
A poset equipped with a grading function is called a \emph{graded poset}.
Of course, not all posets can be turned into graded ones. We will be mainly
interested in finite posets. If $\CP$ is finite and admits a grading function,
there is a rather natural choice for such a function: there exists a unique
grading function $|-|$ with the property that the smallest value it takes
on each connected component is $0$.\footnote{If the poset
has the smallest element $x$, then $|x|=0$. In general, the grading function will be $0$
for at least one minimal element in each connected component,
but not necessarily all of them.}
In the following by a graded poset we mean a finite poset equipped with this
natural grading function.
Moreover, all the posets of interest will contain a smallest element.

\begin{example}\label{ex:yhw}
    Let $\YD_{h,w}$ denote the set of Young diagrams of height at most $h$
    and width at most $w$. This set can be identified with the set of integer
    sequences $(\lambda_1,\lambda_2,\ldots,\lambda_h)$ such that
    $w\geq\lambda_1\geq\lambda_2\geq\cdots\geq\lambda_h\geq 0$.
    There is a natural partial order on $\YD_{h,w}$ given by inclusion of diagrams:
    $\lambda\subseteq \mu$ if $\lambda_i\leq \mu_i$ for all $i=1,\ldots, h$.
    With this partial order the poset $\YD_{h,w}$ is graded, and
    $|\lambda|=\lambda_1+\lambda_2+\cdots+\lambda_h$ equals the number of boxes
    in the diagram $\lambda$.
\end{example}

\subsubsection{Exceptional collections}
Througout the paper we will only work with triangulated categories which
are $\kk$-linear. Let $\CT$ be such a category.
Recall that an object $E\in\CT$ is called
\emph{exceptional} if $\Hom(E, E)=\kk$ and $\Ext^t(E, E)=\Hom(E, E[t])=0$ for
all $t\neq 0$. Let $\CP$ be a poset. An \emph{exceptional collection}
indexed by $\CP$ is a collection of exceptional objects $\{E_x\}_{x\in\CP}$
such that $\Ext^\bullet(E_x, E_y)=0$ unless $x\preceq y$. We denote by
$\langle E_x \mid x\in \CP \rangle$ the smallest strictly full triangulated subcategory
in $\CT$ containing all $E_x$. If $\CP$ is finite, one can always refine the order
so that it becomes isomorphic to the poset $\{1, 2, \ldots, l\}$. Under this
isomorphism we get the usual definition of an exceptional collection: that is,
a~collection of exceptional objects $(E_1,E_2,\ldots, E_l)$ such that
$\Ext^\bullet(E_j, E_i)=0$ for all $l\geq j>i\geq 1$.

We allow ourselves to go over some of the examples mentioned in the introduction
in order to illustrate our grading conventions.

\begin{example}\label{ex:gr-ex}
    Let $V$ be an $n$-dimensional vector space over $\kk$.
    Consider the Grassmannian $\Gr(k, V)$, and let $\CU$ denote the tautological
    rank $k$ subbundle in the trivial bundle $V\otimes \CO$. Kapranov showed in~\cite{Kapranov1984}
    that the bounded derived category of coherent sheaves $D^b(\Gr(k, V))$
    admits a full exceptional collection indexed by~$\YD_{k, n-k}$, the poset introduced
    in Example~\ref{ex:yhw}:
    \begin{equation*}
        D^b(\Gr(k, V)) = \left\langle \Sigma^\lambda \CU^* \mid \lambda \in \YD_{k,n-k}\right\rangle,
    \end{equation*}
    where $\Sigma^\lambda$ denotes the Schur functor associated
    with $\lambda$.\footnote{Our convention for Schur functors is such
    that $\Sigma^{(p)}$ is isomorphic to the $p$-th symmetric power $S^p$,
    so $\Sigma^{(1, 1)}\simeq \Lambda^2$.} The fact that this collection
    is indexed by a~poset gives more information about orthogonality of
    different objects: if $\lambda$ and $\mu$ are incomparable when considered
    as Young diagrams (that~is,
    neither is contained in the other), then both
    $\Ext^\bullet(\Sigma^\lambda \CU^*, \Sigma^\mu \CU^*)=0$ and
    $\Ext^\bullet(\Sigma^\mu \CU^*, \Sigma^\lambda \CU^*)=0$.
\end{example}

\subsubsection{Mutations and duality}
Let $(E, F)$ be an exceptional pair in $\CT$. The \emph{left mutation} $\mathbb{L}_EF$
of $F$ through $E$ is defined by the distinguished triangle
\begin{equation*}
    \mathbb{L}_EF\to \Ext^\bullet(E, F)\otimes E\xrightarrow{\mathrm{ev}} F\to \mathbb{L}_EF[1],
\end{equation*}
where $\mathrm{ev}$ is the evaluation morphism. Similarly, define the \emph{right mutation}
$\mathbb{R}_FE$ via the distinguished triangle
\begin{equation*}
    \mathbb{R}_FE[-1]\to E\xrightarrow{\mathrm{coev}} \Ext^\bullet(E, F)^*\otimes F\to \mathbb{R}_FE,
\end{equation*}
where $\mathrm{coev}$ is the coevaluation morphism. One easily checks that both $(\mathbb{L}_EF, E)$
and $(F, \mathbb{R}_FE)$ are exceptional pairs. Moreover,
$\langle E, F\rangle=\langle \mathbb{L}_EF, E\rangle=\langle F, \mathbb{R}_FE\rangle$.\footnote{A careful
reader might point out that this definition of mutations does not fully agree with the one given
in the introduction for semiorthonormal bases. Namely, they differ by a sign on the Grothendieck group.
We prefer this choice though, since it is more common in the literature.}

\begin{example}
    Consider the Grassmannian $\Gr(k, V)$. Then the structure sheaf and the dual tautological bundle
    form an exceptional pair $(\CO, \CU^*)$. One quickly checks that $\mathbb{L}_{\CO}\CU^*\simeq \CU^\perp$,
    where $\CU^\perp=(V/\CU)^*$.
\end{example}

Mutations of adjacent elements define
an action of the braid group $\mathrm{Br}_l$ on $l$ strands on the set of exceptional collections
in $\langle E_1,E_2,\ldots,E_l\rangle$, see~\cite{Bondal1989}.
For a fixed collection there are two important elements
in the orbit. Namely, the dual collections. The \emph{left dual} collection to $(E_1,E_2,\ldots,E_l)$
is defined as
\begin{equation*}
    (\mathbb{L}_{E_1}\mathbb{L}_{E_2}\cdots\mathbb{L}_{E_{l-1}}E_l,\ 
    \mathbb{L}_{E_1}\mathbb{L}_{E_2}\cdots\mathbb{L}_{E_{l-2}}E_{l-1},\ \ldots,\ 
    \mathbb{L}_{E_1}E_2,\ E_1).
\end{equation*}
We denote it by $(E^\vee_l, E^\vee_{l-1}, \ldots, E^\vee_1)$. The left dual exceptional collection
can be fully characterized by the following three properties:
\begin{enumerate}
    \item $E^\vee_i\in \langle E_1,E_2,\ldots, E_l\rangle$ for all $1\leq i\leq l$,
    \item $\Ext^\bullet(E_i, E_j^\vee)=0$ for all $i\neq j$,
    \item $\Ext^\bullet(E_i, E_i^\vee)=\kk[-i+1]$ for all $1\leq i\leq l$.
\end{enumerate}

\begin{example}\label{ex:pn}
    Let us return to projective spaces.
    The bounded derived category of $\PP(V)$ has a full exceptional
    collection consisting of the line bundles $\langle \CO, \CO(1), \ldots, \CO(n)\rangle$,
    where $n$ is the~dimension of $V$. Its left dual is given by
    $\langle \Omega^n(n), \Omega^{n-1}(n-1), \ldots, \Omega^1(1), \CO\rangle$,
    where $\Omega^i=\Lambda^i\Omega^1_{\PP(V)}$.
\end{example}

Similarly, the \emph{right dual} collection is defined as
\begin{equation*}
    (E_l,\ \mathbb{R}_{E_l}E_{l-1},\ \mathbb{R}_{E_l}\mathbb{R}_{E_{l-1}}E_{l-2},\ \ldots,\ 
    \mathbb{R}_{E_l}\mathbb{R}_{E_{l-1}}\cdots\mathbb{R}_{E_2}E_1),
\end{equation*}
and will be denoted by $(\lvee{E}_l, \lvee{E}_{l-1}, \ldots, \lvee{E}_1)$. It can be fully
characterized by the following conditions:
\begin{enumerate}
    \item $\lvee{E}_i\in \langle E_1,E_2,\ldots, E_l\rangle$ for all $1\leq i\leq l$,
    \item $\Ext^\bullet(\lvee{E}_i, E_j)=0$ for all $i\neq j$,
    \item $\Ext^\bullet(\lvee{E}_i, E_i)=\kk[-l+i]$ for all $1\leq i\leq l$.
\end{enumerate}

Remark that in a given exceptional collection $(E_1,E_2,\ldots,E_l)$ one can
replace any object $E_i$ with its shift $E_i[t]$ for any integer $t$.
For instance, one can introduce an extra shift in the definitions of
the right and left mutations. The first two defining conditions for the
left and right dual collections will not change, while the third condition
will become slightly nicer:
\begin{equation*}
    \Ext^\bullet(\lvee{E}_i, E_i) = \Ext^\bullet(E_i, E_i^\vee) = \kk\quad \text{for any }1\leq i\leq l.
\end{equation*}
This convention is often reasonable, yet even in the case of the projective
space, Example~\ref{ex:pn}, the dual collection will not consist of vector
bundles while the original collection does.

Meanwhile, the definitions we have just given also have a downside.
Imagine that $(E, F)$ is a fully orthogonal pair in $\CT$. Then
$\mathbb{L}_EF\simeq F[-1]$, while $\mathbb{R}_FE\simeq E[1]$.
This is often inconvenient as well. We propose the following lemma-definition,
which is tailored to the case of an exceptional collection indexed by
a finite graded poset $\CP$. The reader will immediately check that
once the collection is linearly ordered, the left dual
differs from the graded left dual by shifts of objects.

\begin{lemma}\label{lm:gr-dual}
    Let $\langle E_x\mid x\in\CP \rangle$ be an exceptional collection
    indexed by a finite graded poset $\CP$. For any $y\in \CP$ there exists
    a unique (up to isomorphism) object $E^\circ_y\in \langle E_x\mid x\in\CP \rangle$ such that
    \begin{enumerate}
        \item $\Ext^\bullet(E_x, E^\circ_y)=0$ for all $x\neq y$,
        \item $\Ext^\bullet(E_y, E^\circ_y)=\kk[-|y|]$.
    \end{enumerate}
    The objects $E^\circ_y$ form an exceptional collection
    with respect to the opposite poset $\CP^\circ$.
    This collection is~called the \emph{graded left dual,} and
    $\langle E^\circ_y\mid y\in\CP^\circ \rangle=\langle E_x\mid x\in\CP \rangle$.
\end{lemma}

\begin{remark}
    If the poset $\CP$ is linearly ordered, then the definition of the
    graded left dual agrees with the definition of the left dual.
\end{remark}

\begin{example}
    In Example~\ref{ex:gr-ex} we have seen that $D^b(\Gr(k, V))$
    admits a full exceptional collection indexed by the poset~$\YD_{k, n-k}$:
    \begin{equation*}
        D^b(\Gr(k, V)) = \left\langle \Sigma^\lambda \CU^* \mid \lambda \in \YD_{k,n-k}\right\rangle,
    \end{equation*}
    where $\CU$ is the tautological rank $k$ subbundle in $V$.
    The dual collection constructed by Kapranov,
    \begin{equation*}
        D^b(\Gr(k, V)) = \left\langle \Sigma^{\lambda^T} \CU^\perp \mid \lambda \in \YD_{k,n-k}\right\rangle,
    \end{equation*}
    where $\CU^\perp=(V/\CU)^*$, and $\lambda^T$ denotes the transposed diagram,
    is a \emph{graded} left dual.
\end{example}

We leave the definition of the graded right dual to the reader, indicating that
for the right dual the grading function should be taken for the opposite poset.

\subsubsection{Categorical lemma}
Let $\CT=\langle E_x\mid x\in\CP\rangle$ be a~proper triangulated category
generated by a~graded exceptional collection.
Denote by $\langle G_x\mid x\in\CP^\circ\rangle$ its graded left dual.
Assume that $F:\CT\to \CT'$ is an~exact functor into another proper triangulated
category $\CT'$. Put $G'_x=F(G_x)$ for all $x\in\CP^\circ$ and assume
that $\langle G'_x\mid x\in\CP^\circ\rangle$ form a graded exceptional
collection in $\CT'$. Denote by $\langle E'_x\mid x\in\CP\rangle$
its graded right dual.

\begin{lemma}\label{lm:adj}
    For all $x, y\in \CP$ one has
    \begin{equation*}
        \Ext^\bullet_{\CT'}(E'_x, F(W_y))\simeq \Ext^\bullet_{\CT}(E_x, E_y).
    \end{equation*}
\end{lemma}
\begin{proof}
    Since the category $\CT$ is saturated, the functor $F$ has a left adjoint $F^*$
    (see~\cite{Bondal1989}). In particular,
    \begin{equation*}
        \Ext^\bullet_{\CT'}(E'_x, F(E_y))\simeq \Ext^\bullet_{\CT}(F^*(E'_x), E_y).
    \end{equation*}
    We claim that $F^*(E'_x)\simeq E_x$.
    On the one hand, $F^*(E'_x) \in \langle E_x\mid x\in\CP\rangle = \langle G_x\mid x\in\CP^\circ\rangle$.
    On the other hand, for any $y\in \CP$ one has
    \begin{equation*}
        \Ext^\bullet_{\CT}(F^*(E'_x), G_y)\simeq
        \Ext^\bullet_{\CT'}(E'_x, G'_y) = \begin{cases}
            \kk[-|x|] & \text{if}\ x = y, \\
            0 & \text{otherwise}.
        \end{cases}
    \end{equation*}
    These are precisely the defining conditions of the graded right dual collection
    to $\langle G_x\mid x\in\CP^\circ\rangle$, which is $\langle E_x\mid x\in\CP\rangle$.
\end{proof}

\subsubsection{Spectral sequence associated with the graded dual}
As stated in the introduction, dual collections provide a particularly nice
computational tool. Let $(E_1,E_2,\ldots, E_l)$ be an exceptional collection
in $\CT$. Recall that a cohomological functor from $\CT$ to an abelian category $\CA$
is an additive functor $F:\CT\to \CA$ which takes distinguished triangles to exact sequences.
As usual, we denote by $F^i$ the composition $F\circ [i]$.

\begin{proposition}[{\cite[Section 2.7.3]{Gorodentsev2004}}] % \label{prop:gor-ss}
    Let $G\in \langle E_1,E_2,\ldots,E_n\rangle$. There is a spectral sequence
    with the first page given by
    \begin{equation}\label{eq:gor-ss}
        E^{p,q}_1 = \bigoplus_{i+j=q}\Ext^{-i}(G, E^\vee_{p+1})^*\otimes F^j(E_{p+1})
    \end{equation}
    converging to $F^{p+q}(G)$.
\end{proposition}

We will be interested in the case where $\CT=D^b(\mathcal{A})$ is the bounded
derived category of an abelian category $\mathcal{A}$ (for instance, the bounded
derived category of coherent sheaves on a smooth projective variety), and
$F=\mathcal{H}^0$ is the usual 0-th cohomology functor.
Assume that the exceptional collection $(E_1,E_2,\ldots, E_l)$ consists of pure objects
(for instance, of coherent sheaves). Then the spectral sequence~\eqref{eq:gor-ss}
simplifies to
\begin{equation}\label{eq:ss-pure}
    E^{p,q}_1 = \Ext^{-q}(G, E^\vee_{p+1})^*\otimes E_{p+1}\ \Rightarrow\ \mathcal{H}^{p+q}(G).
\end{equation}

In the case of an exceptional collection indexed by a graded poset $\CP$ the spectral
sequence~\eqref{eq:ss-pure} becomes
\begin{equation}\label{eq:ss-graded}
    E^{p,q}_1 = \bigoplus_{x\in\CP,\ |x|=p}\Ext^{-q}(G, E^\circ_{x})^*\otimes E_{x}\ \Rightarrow\ \mathcal{H}^{p+q}(G).
\end{equation}

\subsection{Exceptional collections on Lagrangian Grassmannians}\label{ssec:fec-lgr}
Let $V$ be a $2n$-dimensional vector space over $\kk$ equipped with a non-degenerate skew-symmetric
bilinear form. We denote by $\LGr(n, V)$ the Lagrangian Grassmannian of maximal isotropic subspaces
in $V$, and by $\CU$ the tautological rank $n$ bundle on $\LGr(n, V)$. The generator of the Picard
group of $\LGr(n, V)$ is $\CO(1)\simeq \Lambda^n\CU^*$.

The Lagrangian Grassmannian comes with an action of the symplectic group $\bfG=\SP_{2n}$.
We denote by $\sfP$ the set of weakly decreasing integer sequences of length $n$:\footnote{It is actually
the dominant cone of the Levi subgroup $\GL_n$ of the corresponding parabolic in $\bfG$.}
\begin{equation*}
    \sfP=\{\lambda\in \ZZ^n \mid \lambda_1\geq \lambda_2\geq \cdots\geq \lambda_n\}.
\end{equation*}
Given $\lambda\in\sfP$, we denote by $\Sigma^\lambda$ the corresponding Schur functor.
We follow the convention under which $\Sigma^{(p, 0, \ldots, 0)}=S^p$.

\subsubsection{Exceptional blocks} In what follows we will consider subcategories in equivariant
derived categories. Since a collection of equivariant objects generate a full triangualated subcategory
both in the usual and the equivariant derived category, in the latter case we will use an additional
subscript and write $\langle-\rangle_\bfG$.

It is well known that every irreducible equivariant vector bundle on $\LGr(n, V)$ is isomorphic
to $\Sigma^\lambda\CU^*$ for some $\lambda\in\sfP$. Moreover, these form an infinite full exceptional
collection in the equivariant derived category:
\begin{equation*}
    D^b_\bfG(\LGr(n, V)) = \left\langle \Sigma^\lambda\CU^* \mid \lambda\in \sfP^{\circ}\right\rangle_\bfG,
\end{equation*}
where $\sfP$ is treated as an infinite poset with the partial order\footnote{This order is natural
from the combinatorial point of view, yet it differs from the order used in~\cite{Kuznetsov2016}.} given by
\begin{equation*}
    \lambda\preceq \mu\quad \text{if and only if}\quad \lambda_i\leq \mu_i\ \text{for all}\ i=1,\ldots, n,
\end{equation*}
and $\sfP^\circ$ is its opposite.

Any subset $S\subseteq \sfP$ with the induced partial order produces an exceptional collection
$\langle \Sigma^\lambda\CU^* \mid \lambda\in S^{\circ}\rangle_\bfG$.
If $S$ is finite and graded, we denote by
\begin{equation*}
    \langle \CE^\lambda \mid \lambda\in S\rangle_\bfG \quad \text{and} \quad \langle \CF^\lambda \mid \lambda\in S\rangle_\bfG
\end{equation*}
the graded right and left duals in $D^b_\bfG(\LGr(n, V))$ to $\langle \Sigma^\lambda\CU^* \mid \lambda\in S^{\circ}\rangle_\bfG$,
respectively.\footnote{A priori, the objects $\CE^\lambda$ and $\CF^\lambda$ depend on the choice
of both $\lambda$ and $S$.}

Kuznetsov and Polishchuk came up with a very simple (yet hard to check in practice) condition
under which the objects $\CE^\lambda$ form an exceptional collection in the non-equivariant category
(we do not distigush between the objects in the equivariant derived category and their images under
the forgetful functor).
\begin{definition}[{See~\cite[Definition~3.1]{Kuznetsov2016}}] % \label{def:lblock}
    A subset $S\subset \sfP$ is called an \emph{exceptional block} if for all $\lambda,\mu\in S$
    the canonical map
    \begin{equation*}
        \bigoplus_{\nu\in S}\Ext^\bullet_\bfG(\Sigma^\lambda\CU^*, \Sigma^\nu\CU^*)\otimes \Hom(\Sigma^\nu\CU^*, \Sigma^\mu\CU^*)
        \to \Ext^\bullet(\Sigma^\lambda\CU^*, \Sigma^\mu\CU^*)
    \end{equation*}
    is an isomorphism.
\end{definition}

In plain words the block condition says that every extension between a pair of objects can be uniquely decomposed
as a sum of equivariant extensions followed by homomorphisms. What is rather surprising is that even though
the original objects $\Sigma^\lambda \CU^*$ for $\lambda\in S$ almost never form an exceptional collection
in the non-equivariant category, the objects of the right dual do form an exceptional collection in
the~non-equivariant category as long as $S$ is a block.

\begin{proposition}[{See~\cite[Proposition~3.9]{Kuznetsov2016}}]\label{prop:rblock}
    If $S\subset \sfP$ is an exceptional block, then the corresponding right dual objects
    form an exceptional collection in $D^b(\LGr(n, V))$,
    \begin{equation*}
        \left\langle \CE^\lambda \mid \lambda \in S \right\rangle \subset D^b(\LGr(n, V)).
    \end{equation*}
\end{proposition}

It turns out that it is natural to call exceptional blocks \emph{right exceptional blocks,}
and that it is useful to consider left exceptional blocks as well.

\begin{definition}[{See~\cite[Remark~2.12]{Fonarev2022}}]
    A subset $S\subset \sfP$ is called an \emph{left exceptional block} if for all $\lambda,\mu\in S$
    the canonical map
    \begin{equation*}
        \bigoplus_{\nu\in S}\Hom(\Sigma^\lambda\CU^*, \Sigma^\nu\CU^*) \otimes \Ext^\bullet_\bfG(\Sigma^\nu\CU^*, \Sigma^\mu\CU^*) 
        \to \Ext^\bullet(\Sigma^\lambda\CU^*, \Sigma^\mu\CU^*)
    \end{equation*}
    is an isomorphism.
\end{definition}

Recall that both the equivariant and the non-equivariant categories have the dualization functor
compatible with the forgetful functor, sending $E$ to $\RHom(E, \CO)$,
which is an anti-autoequivalence. Since any anti-autoequivalence takes exceptional collections
to exceptional collections (with respect to the opposite order) and left and right (graded) dual
collections to right and left (graded) dual collections respectively, we immediately see that $S$ is a right
exceptional block if and only if $-S$ is a left exceptional block, where
$-S = \{-\lambda \mid \lambda \in S\}$ and $-\lambda = (-\lambda_n, -\lambda_{n-1},\ldots, -\lambda_1)$.
The latter follows from the isomorphism $(\Sigma^\lambda\CU^*)^*\simeq\Sigma^{-\lambda}\CU^*$.
We conclude that the following statement, which is
dual to Proposition~\ref{prop:rblock}, holds.

\begin{proposition} % \label{prop:lblock}
    If $S\subset \sfP$ is a left exceptional block, then the corresponding left dual objects
    form an exceptional collection in $D^b(\LGr(n, V))$,
    \begin{equation*}
        \left\langle \CF^\lambda \mid \lambda \in S \right\rangle \subset D^b(\LGr(n, V)).
    \end{equation*}
    Moreover, $\CF^\lambda \simeq (\CE^{-\lambda})^*$.
\end{proposition}

The following theorem uses the term \emph{semiorthogonal decomposition.} Recall that full triangulated
subcategories $\CT_1,\CT_2\subseteq\CT$ are called semiorthogonal, one writes $\CT_1\subseteq\CT_2^\perp$,
if $\Hom_\CT(X_2, X_1)=0$ for all $X_1\in\CT_1$ and $X_2\in\CT_2$. A \emph{semiorthogonal decomposition}
of a category $\CT$ is a collection of full triangulated subcategories $\CT_1,\CT_2,\ldots,\CT_r$
such that $\CT_i\subseteq \CT_j^\perp$ for all $1\leq i < j\leq r$ and $\CT$ is the smallest strictly
full triangulated subcategory containing all $\CT_i$. The notation for a semiorthogonal decomposition
is $\CT=\langle\CT_1,\CT_2,\ldots,\CT_r\rangle$. Finally, for a full triangulated subcategory
$\CB\in D^b(\LGr(n, V))$ we denote by $\CB(i)$ its image under the autoequivalence given
by tensor product with the line bundle $\CO(i)$.

\begin{theorem}[{See~\cite[Theorem~9.2]{Kuznetsov2016} and \cite[Theorem~4]{Fonarev2022}}]
    The set
    \begin{equation*}
        B_h=\{\lambda\in \sfP \mid n-h\geq \lambda_1\geq\lambda_2\geq\cdots\geq\lambda_h\geq \lambda_{h+1} = \cdots = \lambda_n = 0\}
    \end{equation*}
    is a right exceptional block for all $h=0,\ldots,n$. Moreover, there is semiorthogonal decomposition
    \begin{equation}\label{eq:kpf}
        D^b(\LGr(n, V)) = \left\langle \CB_0, \CB_1(1), \CB_2(2),\ldots, \CB_n(n)\right\rangle,
    \end{equation}
    where $\CB_h = \langle \CE^\lambda_{B_h} \mid \lambda \in B_h\rangle$.
\end{theorem}

Let us make a few remarks. First, $B_h\cong\YD_{h, n-h}$ as (graded) posets.
Second, the object $\CE^\lambda_{B_h}$ for $\lambda\in\YD_{h, n-h}$ has the following
homological description. It is an equivariant object such that
\begin{equation}\label{eq:el}
    \CE^\lambda_{B_h} \in \langle \Sigma^\mu\CU^*\mid \mu \subseteq \lambda\rangle_\bfG \subset D^b_\bfG(\LGr(n, V))
    \quad \text{and} \quad
    \Ext^\bullet_\bfG(\CE^\lambda_{B_h}, \Sigma^\mu\CU^*) = \begin{cases}
        \kk[-|\lambda|] & \text{if}\ \mu=\lambda, \\
        0 & \text{if}\ \mu\subsetneq\lambda.
    \end{cases}
\end{equation}
In particular, it depends only on $\lambda$ (one only needs to know that $\lambda\in\YD_{h,n-h}$
for some $0\leq h \leq n$), so we will abbreviate $\CE^\lambda=\CE^\lambda_{B_h}$.
Finally, by using the dualization anti-autoequivalence we get a semiorthogonal decomposition
\begin{equation*}
    D^b(\LGr(n, V)) = \left\langle \CC_n(-n), \CC_{n-1}(-n+1),\ldots, \CC_1(-1), \CC_0\right\rangle,
\end{equation*}
where $\CC_h = \langle \CF^\lambda_{-B_h} \mid \lambda \in -B_h\rangle$ and $\CF^\lambda$ for $\lambda\in\YD_{h, n-h}$
is an equivariant object which
can be characterized by the following properties:
\begin{equation}\label{eq:fl}
    \CF^\lambda_{-B_h} \in \langle \Sigma^\mu\CU\mid \mu \subseteq \lambda\rangle_\bfG \subset D^b_\bfG(\LGr(n, V))
    \quad \text{and} \quad
    \Ext^\bullet_\bfG(\Sigma^\mu\CU, \CF^\lambda_{-B_h}) = \begin{cases}
        \kk[-|\lambda|] & \text{if}\ \mu=\lambda, \\
        0 & \text{if}\ \mu\subsetneq\lambda.
    \end{cases}
\end{equation}
Again, since $\CF^\lambda_{-B_h}$ depends only on $\lambda$, we will drop the subscript and write $\CF^\lambda$.

\subsubsection{Geometric constructions}
We will need one of the two geometric constructions for the objects $\CF^\lambda$
obtained in~\cite{Fonarev2022}. Since the cases $h=0$ and $h=n$
are trivial ($B_0$ and $B_n$ consist of a single weight $\lambda = (0, 0, \ldots, 0)$,
and $\CF^\lambda=\CE^\lambda=\CO$), let us fix $0 < h < n$. Consider
the diagram
\begin{equation}\label{eq:pqh}
    \begin{tikzcd}[column sep=small]
        & \IFl(h, n; V) \arrow[dl, "p"'] \arrow[dr, "q"] & \\
          \LGr(n, V) & & \IGr(h, V),
        \end{tikzcd}
\end{equation}
where $\IGr(h, V)$ is the Grassmannian of rank $h$ isotropic subspaces
in $V$ and $\IFl(h, n; V)$ is the partial isotropic flag variety.
Denote by $\CW$ the universal rank $h$ bundle on $\IGr(h, V)$
as well as its pullback on $\IFl(h, n; V)$.

\begin{proposition}[{\cite[Lemma~3.4 and Proposition~3.6]{Fonarev2022}}]\label{prop:flgl}
    The bundles $\langle \Sigma^{\mu^T}\CW^* \mid \mu\in \YD_{n-h,h}\rangle$ form an~exceptional
    collection in the non-equivariant derived category $D^b(\IGr(h, V))$.
    Denote its graded left dual by $\langle \CG^\lambda \mid \lambda \in \YD_{n-h, h}^\circ\rangle$.
    Then
    \begin{equation*}
        \CF^\lambda \simeq p_*q^*\CG^\lambda.
    \end{equation*}
\end{proposition}

From now on we will often identify the graded posets $\YD_{h, n-h}$ and $\YD_{n-h,h}$ via transposition
of diagrams.
The graded dual exceptional collection conditions for $\langle \CG^\lambda \mid \lambda \in \YD_{n-h, h}^\circ\rangle$ become
\begin{equation}\label{eq:gl}
    \CG^\lambda \in \langle \Sigma^{\mu}\CW^* \mid \mu\in \YD_{h,n-h}\rangle
    \quad \text{and} \quad
    \Ext^\bullet(\Sigma^{\mu}\CW^*, \CG^\lambda) = \begin{cases}
        \kk[-|\lambda|] & \text{if}\ \mu^T = \lambda, \\
        0 & \text{if}\ \mu\in\YD_{h, n-h}\ \text{and}\ \mu^T\neq \lambda.
    \end{cases}
\end{equation}

\section{Main results}
We continue to use the notation introduced in Section~\ref{ssec:fec-lgr}.
Our first goal is to prove Theorem~\ref{thm:intro-dual}.

\subsection{Dual exceptional collections on Lagrangian Grassmannians}
Recall that characterizations for the objects $\CE^\lambda$ and $\CF^\lambda$
were given in~\eqref{eq:el} and~\eqref{eq:fl} respectively.
The first main result of the paper is~the following theorem.
\begin{theorem}\label{thm:dual}
    Let $0\leq h \leq n$. Then the exceptional collection $\langle \CF^\mu \mid \mu\in \YD_{n-h, h}^\circ \rangle$
    is the graded left dual to $\langle \CE^\lambda \mid \lambda\in \YD_{h, n-h} \rangle$.
    That is, for $\mu\in \YD_{n-h, h}$ one has
    \begin{equation*}
        \CF^\mu \in \langle \CE^\lambda \mid \lambda\in\YD_{h, n-h}\rangle
        \quad \text{and} \quad
        \Ext^\bullet(\CE^\lambda, \CF^\mu) = \begin{cases}
            \kk[-|\lambda|] & \text{if}\ \mu = \lambda^T, \\
            0 & \text{if}\ \mu\in\YD_{n-h, h}\ \text{and}\ \mu\neq \lambda^T.
        \end{cases}
    \end{equation*}
\end{theorem}

The proof of Theorem~\ref{thm:dual} will take the rest of this subsection.
Our strategy is to compare the objects~$\CE^\lambda$ with the graded
right dual exceptional collection to $\langle \CF^\mu \mid \mu\in \YD_{n-h, h}^\circ \rangle$.
Since the cases $h=0$ and $h=n$ are trivial (all the objects considered
are isomorphic to $\CO$), we assume that $0 < h < n$.

\begin{lemma}\label{lm:gen}
    Let $\nu\in \YD_{h,n-h}$. The following subcategories coincide in $D^b(\LGr(n, V))$:
    \begin{equation*}
        \langle \CF^\mu \mid \mu\subseteq \nu^T \rangle = \langle \CE^\lambda \mid \lambda\subseteq \nu \rangle
        = \langle \Sigma^\lambda\CU^* \mid \lambda\subseteq \nu \rangle,
    \end{equation*}
    where $\langle \Sigma^\lambda\CU^* \mid \lambda\subseteq \nu \rangle$ is the smallest strictly full triangulated subcategory in $D^b(\LGr(n, V))$ containing the~corresponding
    objects.
\end{lemma}
\begin{proof}
    Since, the objects $\CE^\lambda$ form a graded right dual exceptional collection to
    $\langle \Sigma^\lambda\CU^* \mid \lambda\subseteq \nu \rangle_{\bfG}$ in~$D^b_\bfG(\LGr(n, V))$,
    one has $\langle \CE^\lambda \mid \lambda\subseteq \nu \rangle_\bfG
    = \langle \Sigma^\lambda\CU^* \mid \lambda\subseteq \nu \rangle_\bfG$.
    Once we apply the forgetful functor
    from $D^b_\bfG(\LGr(n, V))$ to $D^b(\LGr(n, V))$, we see that $\langle \CE^\lambda \mid \lambda\subseteq \nu \rangle
    = \langle \Sigma^\lambda\CU^* \mid \lambda\subseteq \nu \rangle$ in the
    non-equivariant derived category.\footnote{See also the discussion preceding Corollary~3.8 in \cite{Kuznetsov2016}.}
    In a similar fashion one proves that
    \begin{equation*}
        \langle \CF^\mu \mid \mu\subseteq \nu^T \rangle =  \langle \Sigma^\mu\CU \mid \mu\subseteq \nu^T \rangle.
    \end{equation*}
    It remains to show that $\langle \Sigma^\mu\CU \mid \mu\subseteq \nu^T \rangle = \langle \Sigma^\lambda\CU^* \mid \lambda\subseteq \nu \rangle$,
    which follows from~\cite[Lemma~3.7]{Fonarev2022}.
\end{proof}

\begin{proof}[Proof of Theorem~\ref{thm:dual}]

    Consider the isotropic Grassmannian $\IGr(h, V)$ with its tautological
    bundle $\CW$.
    Fix $\nu\in \YD_{h,n-h}$ and consider the graded poset $\CP = \{\lambda \mid \lambda \subseteq \nu\}$.
    By Proposition~\ref{prop:flgl}, the bundles
    $\langle \Sigma^\lambda\CW^* \mid \lambda \in \CP\rangle$
    form an exceptional collection in $D^b(\IGr(h, V))$.
    Let us apply Lemma~\ref{lm:adj} in the~following setup.
    Put $\CT = \langle \Sigma^\lambda\CW^* \mid \lambda \in \CP\rangle$,
    $E_\lambda = \Sigma^\lambda\CW^*$, $G_\lambda = \CG_\lambda$,
    $\CT'=D^b(\LGr(n, V))$,
    and $F = p_*q^*$, where
    $p$ and $q$ are as in~\eqref{eq:pqh}.
    Denote by $\langle \CTE^\lambda \mid \lambda\in \CP \rangle$ the graded right dual
    collection to $\langle \CF^\mu \mid \mu\subseteq \nu^T \rangle$
    in the non-equivariant derived category.
    By Proposition~\ref{prop:flgl}, $\CF^\lambda = F(G_\lambda)$, so Lemma~\ref{lm:adj} is applicable.
    The conclusion is that $\Ext^\bullet(\CTE^\lambda, p_*\Sigma^\mu\CW^*)\simeq \Ext^\bullet(\Sigma^\lambda\CW^*, \Sigma^\mu\CW^*)$.
    A simple Borel--Bott--Weil computation shows that
    $p_*\Sigma^\mu\CW^*\simeq \Sigma^\mu\CU^*$ for $\mu\subseteq \nu$ (see~\cite[Lemma~A.4]{Fonarev2022}).
    Finally, for $\mu\subseteq \nu$ we conclude that
    \begin{equation}\label{eq:cte}
        \Ext^\bullet(\CTE^\nu, \Sigma^\mu\CU^*) = \begin{cases}
            \kk & \text{if}\ \mu=\nu, \\
            0 & \text{if}\ \mu\subsetneq \nu.
        \end{cases}
    \end{equation}

    Let us recall that our goal is to prove that $\CTE^\nu\simeq \CE^\nu$.
    It was shown in~\cite[Corollary~3.8]{Kuznetsov2016} that
    for all $\lambda, \mu\in \CP$ there is an isomorphism
    \begin{equation*}
        \Ext^\bullet(\CE^\lambda, \Sigma^\mu\CU^*) \simeq \Hom(\Sigma^\lambda\CU^*, \Sigma^\mu\CU^*).
    \end{equation*}
    It follows from the Lagrangian Borel--Bott--Weil theorem (see Section~\ref{ssec:lbbw}) that
    \begin{equation*}
        \Hom(\Sigma^\lambda\CU^*, \Sigma^\mu\CU^*) = \begin{cases}
            \kk & \text{if}\ \lambda=\mu, \\
            0 & \text{if}\ \lambda\not\subseteq\mu.
        \end{cases}
    \end{equation*}
    We conclude that
    \begin{equation}\label{eq:ce}
        \Ext^\bullet(\CE^\nu, \Sigma^\mu\CU^*) = \begin{cases}
            \kk & \text{if}\ \mu=\nu, \\
            0 & \text{if}\ \mu\subsetneq \nu.
        \end{cases}
    \end{equation}

    Let $\nu\in\YD_{h, n-h}$. By Lemma~\ref{lm:gen} we know that $\CTE^\nu, \CE^\nu\in \langle \Sigma^\mu\CU^* \mid \mu\subseteq \nu \rangle$.
    Choose a nontrivial element $\phi\in \Ext^\bullet(\CE^\nu, \Sigma^\nu\CU^*)\simeq \kk$.
    It follows from the discussion preceding Corollary~3.8 in~\cite{Kuznetsov2016} that one has
    an exact triangle in $D^b(\LGr(n, V))$ of the form
    \begin{equation*}
        \CE^\nu \xrightarrow{\phi} \Sigma^\nu\CU^*\to C_\phi\to \CE^\nu[1],
    \end{equation*}
    where the cone $C_\phi$ is in $\langle \Sigma^\mu\CU^* \mid \mu\subsetneq \nu\rangle$.
    Applying the functor $\Hom(\CTE^\nu, -)$ to the previous triangle, from~\eqref{eq:cte} we get
    a morphism $\psi:\CTE^\nu\to \CE^\nu$, which lifts a nontrivial $\xi\in \Hom(\CTE^\nu, \Sigma^\nu\CU^*)\simeq \kk$.
    Consider the~cone of $\psi$:
    \begin{equation*}
        \CTE^\nu \xrightarrow{\psi} \CE^\nu\to C_\psi\to \CTE^\nu[1],
    \end{equation*}
    On the one hand, $C_\psi \in \langle \Sigma^\mu\CU^* \mid \mu\subseteq \nu \rangle$ since both
    $\CTE^\nu, \CE^\nu$ belong to this subcategory. On the other hand, it follows from the construction
    of $\psi$ and formulas~\eqref{eq:cte} and~\eqref{eq:ce} that $\Ext^\bullet(C_\psi, \Sigma^\mu\CU^*) = 0$
    for all $\mu\subseteq \nu$. Thus, $C_\psi=0$, and $\psi$ is an isomorphism.
\end{proof}

\begin{remark}
    Theorem~\ref{thm:dual} shows that the semiorthogonal decomposition~\eqref{eq:kpf}
    is self-dual up to a twist. Namely, if one applies the dualization anti-autoequivalence
    followed by the twist by $\CO(n)$,
    the $h$-th block of the resulting decomposition will be generated
    by the objects $\langle (\CE^{\lambda})^* \mid \lambda \in \YD_{n-h, h}\rangle$.
    Since $(\CE^{\lambda})^* \simeq \CF^\lambda$, we see that the exceptional
    collection generating block $n-h$ maps to the graded left dual
    of the exceptional collection generating block $h$. Thus, the decomposition~\eqref{eq:kpf}
    maps to itself.
\end{remark}

Since $\CF^\lambda$ form a graded left dual exceptional collection to $\CE^\lambda$, there is a
spectral sequence associated to it; namely, spectral sequence~\eqref{eq:ss-graded}
takes the following form.

\begin{corollary}
    For any object $G\in \left\langle \CE^\lambda \mid \lambda \in \YD_{h, n-h}\right\rangle$
    there is a spectral sequence of the form
    \begin{equation}\label{eq:ss-el}
        E^{p,q}_1 = \bigoplus_{\lambda\in\YD_{h,n-h},\ |\lambda|=p}\Ext^{-q}\left(G, \CF^{\lambda^T}\right)^*\!\otimes\CE^{\lambda}\ \Rightarrow\ \mathcal{H}^{p+q}(G).
    \end{equation}
\end{corollary}

In the following section we will use this spectral sequence to prove Theorem~\ref{thm:intro-resol}.

\subsection{Resolutions of irreducible equivariant bundles}\label{ssec:res}
In the final section we produce nice resolutions for irreducible equivariant bundles
on $\LGr(n, V)$ of the from $\Sigma^\lambda \CU^*$ for $\lambda\in\YD_{h, n-h}$ for some $0\leq h\leq n$
in terms of bundles of the form $\CE^\mu$. The proof will be given in Section~\ref{sssec:res},
and before we introduce some material required for our computations.

\subsubsection{Balanced diagrams}
Let $\lambda$ be a Young diagram. Most commonly it is represented by a sequence of integers
$\lambda=(\lambda_1,\ldots,\lambda_k)$ such that $\lambda_1\geq\lambda_2\geq \cdots \geq \lambda_k\geq 0$,
where $\lambda_i$ is the length of the $i$-th row of $\lambda$.
There is an alternative description via hook lengths.
Let us say that $\lambda$ \emph{has rank $s$}\footnote{One also says that $\lambda$ has a {\em Durfee square of size $s$}.}
if $\lambda_s\geq s$ and $\lambda_{s+1}\leq s$. Graphically, $s$ is the size of the largest square that
fits into $\lambda$, or the length of the diagonal of $\lambda$. Let $a_i$ and $b_i$ denote
the number of boxes to the right of the $i$-th diagonal box (including itself) and below it (including itself)
respectively. Classically, $a_i$ and $b_i$ are called the \emph{arm} and the \emph{leg length.}
One could alternatively say that $a_i=\lambda_i-(i-1)$ and $b_i=\max \{1\leq j\leq k\mid \lambda_j\geq i\} - (i-1)$.
Equivalently, $b_i=(\lambda^T)_i - (i-1)$.
We will write $\lambda=(a_1, a_2,\ldots, a_s|b_1, b_2,\ldots, b_s)$ if $\lambda$ has rank $s$ and
its arm and leg lengths are $a_i$ and $b_j$, respectively.

\begin{example}
    The diagram $\lambda = (3, 2, 2, 1) = (3, 1|4, 2)$ has rank $2$.
    \begin{equation*}
        \ydiagram[*(lightgray)]{2, 2} *[*(white)]{3, 2, 2, 1}
    \end{equation*}
\end{example}

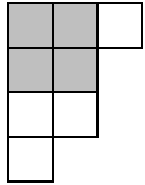
\begin{definition}\label{def:balanced}
    A diagram $\lambda=(a_1, a_2,\ldots, a_s|b_1, b_2,\ldots, b_s)$ is called \emph{balanced} if
    \begin{equation*}
        a_i=b_i+1
    \end{equation*}
    for every
    $1\leq i\leq s$. The set of balanced diagrams with $2t$ boxes is denoted by $\mathrm{B}_{2t}$.\footnote{The number of boxes in a balanced diagram is always even.}
\end{definition}

Balanced diagrams are important to us for the following reason.

\begin{proposition}[{\cite[Proposition~2.3.9]{Weyman2003}}]
    Let $E$ be a locally free module over a commutative ring of characteristic zero.
    Then
    \begin{equation*}
        \Lambda^t\left(S^2 E\right) \simeq \bigoplus_{\lambda\in \mathrm{B}_{2t}}\Sigma^\lambda E.
    \end{equation*}
\end{proposition}

\begin{lemma}\label{lm:balanced}
    A diagram $\lambda=(\lambda_1,\ldots,\lambda_k)$ of rank $s$ is balanced if and only if $\lambda_s\geq s+1$ and
    \begin{equation*}
        (\lambda_1-(s+1), \lambda_2-(s+1), \ldots, \lambda_s-(s+1))^T = (\lambda_{s+1},\lambda_{s+2},\ldots,\lambda_{k}),
    \end{equation*}
    where, in particular, we assume that $\lambda_s\geq s+1$.
\end{lemma}
\begin{proof}
    Remark that $(a_1, a_2,\ldots, a_s|b_1, b_2,\ldots, b_s)^T=(b_1, b_2,\ldots, b_s|a_1, a_2,\ldots, a_s)$.
    In particular, $\lambda$ is \emph{symmetric}, $\lambda=\lambda^T$, if and only if $a_i=b_i$ for all $1\leq i\leq s$.
    It follows from Definition~\ref{def:balanced} that $\lambda$ is balanced if and only if
    $\lambda_s\geq s+1$ and $\mu=(\lambda_1-1, \lambda_2-1,\ldots, \lambda_s-1, \lambda_{s+1},\ldots, \lambda_{k-1},\lambda_k)$
    is symmetric. Assume the~latter.
    
    The rank of $\mu$ equals that of $\lambda$, which is $s$. One can separate $\mu$ into three
    parts: a square of size~$s$, everything to the right of it, and everything below it.
    The latter two parts are nothing but the diagrams $\mu^r = (\mu_1-s,\mu_2-s,\ldots,\mu_s-s)$ and $\mu^b = (\mu_{s+1},\mu_{s+2},\ldots,\mu_k)$.
    Under transposition the square maps to itself, while $\mu^r$ and $\mu^b$ get interchanged and transposed:
    $(\mu^T)^r=(\mu^b)^T$ and $(\mu^T)^b=(\mu^r)^T$. One immediately concludes that $\mu$ is symmetric
    if and only if $(\mu^r)^T = \mu^b$.
\end{proof}

\subsubsection{Lagrangian Borel–Bott–Weil}\label{ssec:lbbw}
The celebrated Borel--Bott--Weil theorem fully describes the cohomology of irreducible equivariant line
bundles on rational homogeneous varieties. Since we only use it in one place, we formulate
a much simplified version of it and refer the interested reader to~\cite[Chapter~4]{Weyman2003}.

Given a weakly decreasing sequence $\lambda\in\ZZ^n$, we denote by $-\lambda$ the sequence
$(-\lambda_n, -\lambda_{n-1}, \ldots, -\lambda_1)$, which is again weakly decreasing.
The sum of weakly decreasing sequences is defined termwise.
A weakly decreasing sequence is called \emph{non-singular} if the absolute values of all of its
terms are positive and distinct. If $\lambda$ is non-singular, we denote by $\|\lambda\|$
the sequence obtained by taking all the absolute values of the terms of $\lambda$
and writing them in decreasing order. If $\lambda$ is non-singular,
then $\|\lambda\|-\rho$ is a weakly decreasing sequence with non-negative terms,
where $\rho=(n, n-1, \ldots, 1)$. The set of weakly decreasing sequences
$\mu\in\ZZ^n$ with non-negative terms is identified with the set of dominant
weights of $\SP(V)$, where $V$ is a $2n$-dimensional symplectic vector space.
We denote by $V^{\langle\mu\rangle}$ the corresponding irreducible
representation of $\SP(V)$. For instance, if $\mu=(0, 0,\ldots, 0)$, then $V^{\langle\mu\rangle}=\kk$.

\begin{theorem}[Lagrangian Borel--Bott--Weil]\label{thm:lbbw}
    Let $\lambda\in\ZZ^n$ be a weakly decreasing sequence. If $-\lambda+\rho$ is
    non-singular, then
    \begin{equation*}
        H^\bullet(\LGr(n, V), \Sigma^\lambda\CU) = V^{\langle\|-\lambda+\rho\|-\rho \rangle}[-\ell],
    \end{equation*}
    where $\ell$ equals the number of negative terms in $-\lambda+\rho$ plus the number
    of pairs $1\leq i < j\leq n$ such that $(-\lambda+\rho)_i + (-\lambda+\rho)_j < 0$.
    Otherwise, $H^\bullet(\LGr(n, V), \Sigma^\lambda\CU) = 0$.
\end{theorem}

\subsubsection{Vanishing lemma}
We will need the following result, which is definitely well known to experts.
A~proof is included for the sake of completeness.
\begin{lemma} % \label{lm:lgr-van}
    Let $\lambda\in\YD_{n,n+1}$ be a Young diagram.
    \begin{enumerate}
        \item \label{itm:lgr-van-0} If $\lambda$ is not balanced, then $H^\bullet(\LGr(n, V), \Sigma^\lambda\CU) = 0$.
        \item \label{itm:lgr-van} If $\lambda$ is balanced and $|\lambda|=2t$, then $H^\bullet(\LGr(n, V), \Sigma^\lambda\CU) = \kk[-t]$.
    \end{enumerate}
\end{lemma}

\begin{proof}
    The proof is a direct application of Theorem~\ref{thm:lbbw}.
    In~order to show~(\ref{itm:lgr-van-0}), we need to show that the~weight $-\lambda+\rho$
    is non-singular if and only if $\lambda$ is balanced. Consider the strictly decreasing
    sequence
    \begin{equation}\label{eq:rho-lambda}
        -\lambda+\rho=(n-\lambda_n,\ (n-1)-\lambda_{n-1},\ \ldots,\ 2-\lambda_2,\ 1-\lambda_1).
    \end{equation} 
    Since $0\leq \lambda_i\leq n+1$ for all $i$, all the absolute values of the terms of
    $-\lambda+\rho$ are between $0$ and $n$. Recall that $-\lambda+\rho$ is non-singular
    if and only if all the absolute values of its terms are distinct and positive.
    Assume the latter. Let $s$ denote the rank of $\lambda$. Then the first $n-s$
    terms of~\eqref{eq:rho-lambda} are positive, while the last $s$ terms are non-positive.
    Since the absolute values of the latter cannot be zero (the weight being non-singular),
    we conclude that $\lambda_i\geq s+1$ for $i=1,\ldots,s$.
    Put $\mu_i=\lambda_i-(s+1)$. Then the sequence
    \begin{equation}\label{eq:comb-gr-dual}
        (n-\lambda_n,\ (n-1)-\lambda_{n-1},\ \ldots,\ (s+1)-\lambda_{s+1},\ s + \mu_1,\ (s-1) + \mu_2,\ \ldots,\ 1+\mu_s)
    \end{equation}
    differs from~\eqref{eq:rho-lambda} only in the last $s$ terms: their signs have been changed,
    and their order has been reversed.
    Remark that $(\lambda_{s+1}, \lambda_{s+2}, \ldots, \lambda_n)=\lambda^b\in \YD_{n-s,s}$, while
    $\mu\in \YD_{s, n-s}$. The sequence~\eqref{eq:comb-gr-dual} is very well known to anyone who
    has ever studied Kapranov's work. It follows from \cite[Section~2.7]{Kapranov1984} that all the terms
    in~\eqref{eq:comb-gr-dual} are distinct if and only if $\lambda^b=\mu^T$. The latter is equivalent,
    by Lemma~\ref{lm:balanced}, to saying that $\lambda$ is balanced.

    Let us turn to~(\ref{itm:lgr-van}). Assume that $\lambda$ is balanced. Since the absolute values of the terms
    of $-\lambda+\rho$ take all the values between $1$ and $n$, by the Borel--Bott--Weil theorem the only nonzero
    cohomology will be equal to $\kk$. We only need to determine the degree in which it sits, and the latter
    equals the number of negative terms in $-\lambda+\rho$ plus the number of pairs $1\leq i < j\leq n$ such
    that $(-\lambda+\rho)_i + (-\lambda+\rho)_j < 0$. The first number equals $s$, and we need to compute the second
    number. Remark that if $1\leq j \leq n-s$, then both $(-\lambda+\rho)_i$ and $(-\lambda+\rho)_j$ are positive.
    Thus, we may assume that $n-s < j \leq n$. If $n-s < i\leq n$, then any $i<j\leq n$ contributes to the count
    since both terms are negative. There are $\frac{s(s-1)}{2}$ such pairs. Finally, assume that
    $1\leq i\leq n-s$ and $n-s < j\leq n$. Since $(-\lambda+\rho)_j < 0$, we need to determine when
    $(-\lambda+\rho)_i < -(-\lambda+\rho)_j$. Such pairs correspond precisely to inversions in
    the sequence~\eqref{eq:comb-gr-dual}.
    It is easy to show that the number of those equals $|\lambda^b|$ (see~\cite[Section~2.7]{Kapranov1984}).
    We conclude that for a balanced $\lambda$ the only nontrivial cohomology sits in degree
    $s + \frac{s(s-1)}{2} + |\lambda^b| = \frac{s(s+1)}{2} + |\lambda^b|$. It remains to recall
    that $\lambda$ consists of $\lambda^b$, $\mu=(\lambda^b)^T$, and a rectangle of size $s\times (s+1)$,
    so $|\lambda| = |\lambda^b| + |\mu| + s(s+1)=2|\lambda^b| + s(s+1)$.
\end{proof}

The following is a simple relative version of the previous lemma. As usual, all the functors are derived.

\begin{corollary}\label{cor:van}
    Let $X$ be smooth projective variety, and let $\CV$ be a rank $2k$ symplectic bundle on $X$.
    Consider the relative Lagrangian Grassmannian $\pi:\LGr_X(k, \CV)\to X$. Denote by
    $\CU\subset \pi^*\CV$ the tautological bundle. Let $\nu \in \YD_{k, k+1}$.
    \begin{enumerate}
        \item If $\nu$ is not balanced, then $\pi_*\Sigma^\nu\CU=0$.
        \item If $\nu$ is balanced and $|\nu|=2t$, then $\pi_*\Sigma^\nu\CU = \CO_X[-t]$.
    \end{enumerate}
\end{corollary}

\subsubsection{Skew Schur functors}\label{ssec:skew}
Before we prove the second main result of the paper, we need to say a few
words about skew Young diagrams. Let $\lambda$ and $\mu$ be two Young diagrams such that
$\mu\subseteq \lambda$. Then one can define the \emph{skew Schur functor} $\Sigma^{\lambda/\mu}$,
see~\cite[Section~2.1]{Weyman2003}. The definition itself is not important for the present
paper. What is important is the decomposition~\eqref{eq:skew-schur}.

If $E$ is a free module over a commutative ring of
characteristic zero and $\alpha$ and $\beta$ are two Young diagrams,
one has a direct sum decomposition
\begin{equation*}
    \Sigma^\alpha E\otimes \Sigma^\beta E\simeq \bigoplus_{|\kappa|=|\alpha|+|\beta|}\left(\Sigma^\kappa E\right)^{\oplus c(\alpha,\beta;\kappa)},
\end{equation*}
where $c(\alpha,\beta;\kappa)$ are the celebrated Littlewood--Richardson coefficients.
It turns out that one has a direct sum decomposition for skew Schur functors
controlled by the same coefficients. Namely,
\begin{equation}\label{eq:skew-schur}
    \Sigma^{\lambda/\mu}E \simeq \bigoplus_{|\nu| + |\mu| = |\lambda|}\left(\Sigma^\nu E\right)^{\oplus c(\nu, \mu; \lambda)},
\end{equation}
see~\cite[Theorem~2.3.6]{Weyman2003}.

If $c(\nu, \mu; \lambda)>0$, we will write $\nu\subseteq \lambda/\mu$.
Remark that $\nu\subseteq \lambda/\mu$ implies that $\nu\subseteq \lambda$.
In particular, if~$\lambda\in \YD_{h,w}$, then $\nu\subseteq \lambda/\mu$ implies $\nu\in\YD_{h, w}$.

Inspired by decomposition~\eqref{eq:skew-schur}, for a pair
of diagrams $\lambda, \mu\in \YD_{h,w}$, where $h+w=n$ and $\mu\subseteq\lambda$, we~put
\begin{equation}\label{eq:skew-el}
    \CE^{\lambda/\mu} = \bigoplus_{\nu\subseteq \lambda/\mu}\left(\CE^\nu\right)^{\oplus c(\nu, \mu; \lambda)}.
\end{equation}
We extend the definition, as usual, by setting $\CE^{\lambda/\mu} = 0$ for $\mu\nsubseteq\lambda$.

\subsubsection{Resolutions}\label{sssec:res}
We are ready to present the main application of Theorem~\ref{thm:dual}.

\begin{theorem}\label{thm:resol}
    Let $\lambda\in\YD_{h,n-h}$ for some $0\leq h\leq n$. There is an exact $\SP(V)$-equivariant sequence of vector bundles
    on $\LGr(n, V)$ of the form
    \begin{equation}\label{eq:resol}
        0 \to \bigoplus_{\mu\in \mathrm{B}_{h(h+1)}}\CE^{\lambda/\mu}\to
        \cdots
        \to \bigoplus_{\mu\in \mathrm{B}_{2t}}\CE^{\lambda/\mu}
        \to \cdots\to
        \bigoplus_{\mu\in \mathrm{B}_4}\CE^{\lambda/\mu} \to
        \bigoplus_{\mu\in \mathrm{B}_2}\CE^{\lambda/\mu} \to
        \CE^\lambda\to \Sigma^\lambda\CU^*\to 0,
    \end{equation}
    where $\mathrm{B}_{2t}$ denotes the set of balanced diagrams with $2t$ boxes.
\end{theorem}

\begin{proof}
    The vector bundle $\Sigma^\lambda\CU^*$ lies in the subcategory
    $\langle \CE^\lambda \mid \lambda \in \YD_{h, n-h}\rangle$. By Theorem~\ref{thm:dual},
    the collection $\langle \CF^{\mu^T} \mid \mu\in \YD_{h, n-h}\rangle$ is the graded left
    dual to the exceptional collection $\langle \CE^\lambda \mid \lambda \in \YD_{h, n-h}\rangle$, and the~corresponding
    spectral sequence~\eqref{eq:ss-el} is of the form
    \begin{equation}\label{eq:ss-main}
        E^{p,q}_1 = \bigoplus_{\alpha\in\YD_{h,n-h},\ |\alpha|=p}\Ext^{-q}\left(\Sigma^\lambda\CU^*, \CF^{\alpha^T}\right)^*\!\otimes\CE^{\alpha}\ \Rightarrow\ \mathcal{H}^{p+q}(\Sigma^\lambda\CU^*).
    \end{equation}
    
    Let us compute
    $\Ext^\bullet(\Sigma^\lambda\CU^*, \CF^{\mu})$ for $\mu=\alpha^T\in \YD_{n-h, h}$ using
    the isomorphism $\CF^\mu\simeq p_*q^*\CG^\mu$ given in Proposition~\ref{prop:flgl},
    where $p$ and $q$ come from the diagram
    \begin{equation*}
        \begin{tikzcd}[column sep=small]
            & \IFl(h, n; V) \arrow[dl, "p"'] \arrow[dr, "q"] & \\
              \LGr(n, V) & & \IGr(h, V)
        \end{tikzcd}
    \end{equation*}
    In the following we often simply write $\CU$ instead of $p^*\CU$
    and $\Sigma^\lambda\CU^*$ instead of $p^*\Sigma^\lambda\CU^*$.
    One has a~sequence of isomorphisms
    \begin{align*}
        \Ext^\bullet(\Sigma^\lambda\CU^*, \CF^\mu) & \simeq \Ext^\bullet(\Sigma^\lambda\CU^*, p_*q^*\CG^\mu)\\
        & \simeq \Ext^\bullet(\Sigma^\lambda\CU^*, q^*\CG^\mu) \\
        & \simeq H^\bullet(\IFl(h, n; V), \Sigma^\lambda\CU\otimes q^*\CG^\mu) \\
        & \simeq H^\bullet(\IGr(h; V), q_*(\Sigma^\lambda\CU)\otimes \CG^\mu),
    \end{align*}
    where the second isomorphism is given by adjunction, and the last one follows from the projection formula.
    As usual, all the functors are derived.
    
    Consider the spectral sequence associated with the composition of derived functors. Its second page is given by
    \begin{equation}\label{eq:ss-rq}
        H^i(\IGr(h; V), R^jq_*(\Sigma^\lambda\CU)\otimes \CG^\mu)\quad \Rightarrow\quad H^{i+j}(\IGr(h; V), q_*\Sigma^\lambda\CU\otimes \CG^\mu).
    \end{equation}

    Let us compute $R^jq_*(\Sigma^\lambda\CU)$. Recall that we denoted by $\CW\subset \CU$ the universal flag on $\IFl(h, n; V)$.
    There is a filtration of length $|\lambda|$ on $\Sigma^\lambda\CU$ associated with the short exact sequence $0\to \CW\to \CU\to \CU/\CW\to 0$,
    whose $k$-th associated quotient is isomorphic to
    \begin{equation}\label{eq:filt}
        \bigoplus_{\nu\subseteq \lambda,\,|\nu|=k}\Sigma^{\lambda/\nu}\CW\otimes \Sigma^\nu(\CU/\CW),
    \end{equation}
    where if the height of $\nu$ is greater than $n-h$, then $\Sigma^\nu(\CU/\CW)=0$
    by convention.

    Since for any $\nu\subseteq \lambda$ the bundle $\Sigma^{\lambda/\nu}\CW$ is pulled back from $\IGr(h, V)$,
    we conclude by the projection formula that 
    \begin{equation}\label{eq:proj}
        R^jq_*(\Sigma^{\lambda/\nu}\CW\otimes \Sigma^\nu(\CU/\CW)) \simeq \Sigma^{\lambda/\nu}\CW\otimes R^jq_*\Sigma^\nu(\CU/\CW).
    \end{equation}
    
    Recall that $\IFl(h, n; V)$ is the
    relative Lagrangian Grassmannian $\LGr(n-h, \CW^\perp/\CW)$ over $\IGr(h, V)$. Since $\nu\subseteq \lambda$,
    one has $\nu\in\YD_{h,n-h}$. 
    Thus, $\Sigma^{\lambda/\nu}\CW\otimes R^jq_*\Sigma^\nu(\CU/\CW)=0$ unless $\nu\in\YD_{n-h, n-h}$.
    Finally, if~$\nu\in\YD_{n-h, n-h}$, then we are in the situation when we can apply Corollary~\ref{cor:van}.
    We conclude that if $\nu\subseteq\lambda$, then
    \begin{equation}\label{eq:pfnu}
        \Sigma^{\lambda/\nu}\CW\otimes R^\bullet q_*\Sigma^\nu(\CU/\CW) \simeq \begin{cases}
            \Sigma^{\lambda/\nu}\CW[-t] & \text{if}\ \nu\in B_{2t}, \\
            0 & \text{otherwise.}
        \end{cases}
    \end{equation}

    Using~\eqref{eq:proj} and~\eqref{eq:pfnu}, we can compute $R^jq_*(\Sigma^\lambda\CU)$. Indeed, the spectral sequence associated with
    the~filtration with the quotients~\eqref{eq:filt} degenerates in the first page, and
    \begin{equation*} % \label{eq:pfl}
        R^jq_*(\Sigma^\lambda\CU) \simeq \bigoplus_{\nu\subseteq \lambda,\, \nu\in {\rm B}_{2j}}\Sigma^{\lambda/\nu}\CW.
    \end{equation*}

    We are ready to get back to the spectral sequence~\eqref{eq:ss-rq}. We have
    \begin{align*}
        H^i(\IGr(h; V), R^jq_*(\Sigma^\lambda\CU)\otimes \CG^\mu)
        & \simeq \bigoplus_{\nu\subseteq \lambda,\, \nu\in {\rm B}_{2j}} H^{i}(\IGr(h; V), \Sigma^{\lambda/\nu}\CW\otimes \CG^\mu) \\
        & \simeq \bigoplus_{\nu\subseteq \lambda,\, \nu\in {\rm B}_{2j}} \Ext^{i}(\Sigma^{\lambda/\nu}\CW^*, \CG^\mu) \\
        & \simeq \bigoplus_{\nu\subseteq \lambda,\, \nu\in {\rm B}_{2j}} \bigoplus_{\kappa\subseteq\lambda/\nu} \Ext^{i}((\Sigma^{\kappa}\CW^*)^{\oplus c(\kappa,\nu;\lambda)}, \CG^\mu),
    \end{align*}
    where the last isomorphism comes from~\eqref{eq:skew-schur}. Since $\kappa\subseteq \lambda/\nu$ implies
    that $\kappa\subseteq \lambda$, using~\eqref{eq:gl} we see that
    $\Ext^{i}(\Sigma^{\kappa}\CW^*, \CG^\mu)$ is equal to $\kk$ if $\kappa=\mu^T$ and $i=|\mu|$,
    and is $0$ otherwise.
    Finally,
    \begin{equation*}
        H^{|\mu|}(\IGr(h; V), R^jq_*(\Sigma^\lambda\CU)\otimes \CG^\mu) \simeq
        \bigoplus_{\nu\in B_{2j}}\kk^{\oplus c(\nu, \mu; \lambda)},
    \end{equation*}
    where $2j=|\nu|=|\lambda|-|\mu|$,
    are the only potentially nontrivial cohomology groups, and the spectral sequence~\eqref{eq:ss-rq} degenerates.
    We conclude that
    \begin{equation*}
        H^{|\mu|+j}(\IGr(h; V), q_*\Sigma^\lambda\CU\otimes \CG^\mu) = \bigoplus_{\nu\in B_{2j}}\kk^{\oplus c(\nu, \mu; \lambda)}.
    \end{equation*}

    Let us return to spectral sequence~\eqref{eq:ss-main}. From the computations above we know that the only nontrivial
    terms in it are
    \begin{equation}\label{eq:epq-main}
        E^{-|\lambda|+t, |\lambda|-2t}_1 \simeq \bigoplus_{\nu\in B_{2t}}\bigoplus_{\substack{\mu\subseteq\lambda/\nu\\ |\mu|=|\lambda|-2t}}\left(\CE^{\mu}\right)^{\oplus c(\nu, \mu; \lambda)}
        = \bigoplus_{\nu\in B_{2t}} \CE^{\lambda/\nu},
    \end{equation}
    where the last equality comes from the definition in~\eqref{eq:skew-el}.
    (One might suspect that there is merely a~filtration with the given associated graded quotient;
    however, the right hand side is a direct
    sum of~objects of the form $\CE^\alpha$ with $\alpha\subseteq \lambda$ and $|\alpha| = |\lambda|-2t$,
    and all such objects are orthogonal in the derived category since the elements in the corresponding poset
    are incomparable.)
    We conclude that the spectral sequence
    contains at most one non-trivial term in each diagonal.
    
    As the spectral sequence converges to $\CH^\bullet(\Sigma^\lambda\CU^*)$, one must have a long exact sequence of the form
    \begin{equation*}
        \cdots \to E^{-|\lambda|+t, |\lambda|-2t}_1\to \cdots \to E^{-|\lambda|+2, |\lambda|-4}_1 \to E^{-|\lambda|+1, |\lambda|-2}_1 \to E^{-|\lambda|, |\lambda|}_1
        \to \Sigma^\lambda\CU^* \to 0.
    \end{equation*}
    Since any balanced diagram contained in $\lambda$ has height at most $h$,
    its width is at most $h+1$ and $E^{-|\lambda|+t, |\lambda|-2t}_1=0$ for $t>h(h+1)/2$.
    It remains to use~\eqref{eq:epq-main} to get the desired long exact sequence~\eqref{eq:resol}.
\end{proof}

\begin{remark}
    One can also obtain resolutions~\eqref{eq:resol} working in the equivariant derived category.
\end{remark}

\begin{example}
    We present some examples of resolutions established in Theorem~\ref{thm:resol}.

    \textbf{$\bm{h=1}$:} In this case $\lambda$ consists of just a single row of length at most $n-1$.
    If $\lambda = (0)$, then $\CE^\lambda\simeq\CO$. If $\lambda = (1)$, then $\CE^\lambda\simeq\Sigma^\lambda\CU^*=\CU^*$.
    The interesting case is $\lambda=(p)$ for some $2\leq p \leq n-1$. Since $(2)$ is the only nontrivial
    balanced diagram of height one, and $(p)/(2)$ = $(p-2)$, resolution~\eqref{eq:resol} takes form
    \begin{equation*}
        0\to \CE^{(p-2)}\to \CE^{(p)}\to S^p\CU^* \to 0.
    \end{equation*}
    As a consequence, we see that $\CE^{(p)}$ has a filtration with the associated graded quotients
    of the form $S^{p-2t}\CU^*$ for all $0\leq t\leq p/2$.

    \textbf{$\bm{w=1}:$} In this case $\lambda$ consists of a column of length at most $n-1$.
    Since there are no nontrivial balanced diagrams of width $1$, we conclude that for $\lambda = (t)^T$
    there is an isomorphism 
    \begin{equation*}
        \CE^{(t)^T}\simeq \Lambda^t\CU^*.
    \end{equation*}

    \textbf{$\bm{\lambda=(3,1)}$:} This is the first case when resolution~\eqref{eq:resol} has length greater than $1$.
    Remark that $(2)\in \mathrm{B}_2$ and $(3,1)\in \mathrm{B}_4$ are the only balanced diagrams contained in $\lambda$.
    Moreover, the only nontrivial Littlewood--Richardson coefficients for $(3,1)/(2)$ are
    $c\left((2), (2); \lambda\right)=c\left((1, 1), (2); \lambda\right) = 1$.
    Thus, one has a resolution of the form
    \begin{equation*}
        0\to \CO \to \CE^{(2)}\oplus \CE^{(1, 1)}\to \CE^{(3, 1)}\to \Sigma^{(3, 1)}\CU^*\to 0.
    \end{equation*}
    Since $\CE^{(1,1)}\simeq \Lambda^2\CU^*$ and $\CE^{(2)}$ fits into a short exact sequence
    $0\to \CO\to \CE^{(2)}\to S^2\CU^*\to 0$, with a little bit of work one can show that
    $\CE^{(3,1)}$ is an extension of the form
    \begin{equation*}
        0\to S^2\CU^* \oplus \Lambda^2\CU^* \to \CE^{(3,1)} \to \Sigma^{(3,1)}\CU^* \to 0.
    \end{equation*}
\end{example}

\printbibliography
% \bibliographystyle{plain}
% \bibliography{biblib}

\end{document}